\documentclass[final]{siamltex}

\usepackage{amsmath}
\usepackage{amssymb}
\usepackage{graphicx}
\usepackage{listings}

\usepackage{float}
\usepackage{verbatim}

\newcommand{\dx}{h}
\newcommand{\dv}{h}
\newcommand{\dt}{\tau}
\newcommand{\order}{\ell}
\newcommand{\e}{\textnormal{e}}

\title{Convergence analysis of a discontinuous Galerkin/Strang splitting approximation for the Vlasov--Poisson equations}

\newcommand*\samethanks[1][\value{footnote}]{\footnotemark[#1]}
\author{
Lukas Einkemmer\thanks{Department of Mathematics, University of Innsbruck, Technikerstra\ss e 13, Innsbruck, Austria ({\tt lukas.einkemmer@uibk.ac.at}, {\tt alexander.ostermann@uibk.ac.at}).
The first author was supported by a scholarship of the Vizerektorat f\"ur Forschung, University of Innsbruck.} \and Alexander Ostermann\samethanks
}

\begin{document}

  \maketitle

  \begin{abstract}
  A rigorous convergence analysis of the Strang splitting algorithm with a discontinuous Galerkin approximation in space for the Vlasov--Poisson equations is provided. It is shown that under suitable assumptions the error is of order $\mathcal{O}\left(\dt^2+\dx^q +\dx^q / \dt \right)$, where $\dt$ is the size of a time step, $\dx$ is the cell size, and $q$ the order of the discontinuous Galerkin approximation. In order to investigate the recurrence phenomena for approximations of higher order as well as to compare the algorithm with numerical results already available in the literature a number of numerical simulations are performed.
  \end{abstract}

  \begin{keywords}
  Strang splitting, discontinuous Galerkin approximation, convergence analysis, Vlasov--Poisson equations, recurrence
  \end{keywords}

  \begin{AMS}
  65M12, 82D10, 65L05, 65M60
  \end{AMS}

  \pagestyle{myheadings}
  \thispagestyle{plain}
  \markboth{L.~EINKEMMER AND A.~OSTERMANN}{CONVERGENCE ANALYSIS STRANG SPLITTING \& DISCONTINUOUS GALERKIN}

  \section{Introduction}

    In astro- and plasma physics the behavior of a collisionless plasma is modeled by the Vlasov equation (see e.g.~\cite{Belli:2006})
    \begin{equation}\label{eq:intro-vlasov}
     \partial_t f(t,\boldsymbol{x},\boldsymbol{v}) + \boldsymbol{v} \cdot \nabla f(t,\boldsymbol{x},\boldsymbol{v}) + \boldsymbol{F} \cdot \nabla_{\boldsymbol{v}} f(t,\boldsymbol{x},\boldsymbol{v})=0,
    \end{equation}
    a kinetic model that in certain applications is also called the collisionless Boltzmann equation. It is posed in a $3+3$ dimensional phase space, where $\boldsymbol{x}$ denotes the position and $\boldsymbol{v}$ the velocity. The density function $f$ is the sought-after particle-probability distribution, and the (force) term $\boldsymbol{F}$ describes the interaction of the plasma with the electromagnetic field.

    In this paper we will study the convergence properties of a full discretization of the so called Vlasov--Poisson equations, where the force term
    \begin{equation*}
    \boldsymbol{F} = -\nabla \phi
    \end{equation*}
    is the gradient of the self-consistent electric potential $\phi$. This simplified model is used in various applications, e.g.~in the context of Landau damping.

    For the numerical solution of \eqref{eq:intro-vlasov}, various methods have been considered in the literature, for example particle methods and in particular the particle-in-cell method, see~\cite{Fahey:2008, Filbet:2001, Heath:2011}. Another prevalent approach consists in employing splitting methods, first proposed in the context of the Vlasov--Poisson equations by \cite{Cheng:1976} and later extended to the full Vlasov--Maxwell equations in \cite{Mangeney:2002}. Both papers use second-order Strang splitting. In the seminal paper \cite{Jahnke00}, the convergence properties of Strang-splitting for evolution equations were analyzed with the help of the variation-of-constants formula. This approach was recently extended to Vlasov-type equations in \cite{time_analysis}. In \cite{Besse:2005, Besse:2008, Respaud:2011} semi-Lagrangian methods are combined with Strang splitting. Convergence is shown in the case of the 1+1 dimensional Vlasov--Poisson equations, and in \cite{Bostan:2009} for a special case of the one-dimensional Vlasov--Maxwell equation. In these papers usually Hermite or spline interpolation is employed.

    On the other hand, discontinuous Galerkin approximations in space have been studied for the Vlasov--Poisson equations as well. In \cite{Heath:2007} and \cite{Heath:2011} the weak version of the Vlasov-Poisson equations is discretized by a discontinuous Galerkin scheme and integrated in time by Runge--Kutta methods. In \cite{Rossmanith:2011} a higher order semi-Lagrangian method in time is combined with a discontinuous Galerkin approximation in space. However, no convergence analysis is given. A direct Strang splitting scheme with a discontinuous Galerkin approximation is implemented in \cite{Mangeney:2002}. Since only a single value per cell is advanced in time this leads to a Van Leer scheme. In the before mentioned paper a numerical study of this scheme is conducted.

    In this paper, we will extend the analysis done in \cite{time_analysis} for the Vlasov--Poisson equations in 1+1 dimensions to the fully discretized case. More precisely, we will show that the (direct) Strang splitting scheme combined with a discontinuous Galerkin discretization in space is convergent. Our main result is stated in Theorem~\ref{thm:convergence} below. In addition, we will discuss some numerical aspects of our discretization in section \ref{sec:numerical_simulation}.

  \section{Vlasov--Poisson equations in 1+1 dimensions} \label{sec:vp11}

    In this section we perform the convergence analysis of Strang splitting in time with a discontinuous Galerkin approximation in space for the Vlasov--Poisson equations in 1+1 dimensions. To that end we first describe the setting as well as give the necessary regularity results (sections \ref{sec:setting} to \ref{sec:regularity}). We then describe the time (section \ref{sec:time_discretization}) and space discretization (sections \ref{sec:space_discretization} and \ref{sec:translation_projection}). In section \ref{sec:jackson_for_discontinuous} we will extend a commonly employed approximation result
    from $\mathcal{C}^{\order+1}$ functions to piecewise polynomials with a small jump discontinuity. This extension is crucial to show consistency (which is done in section \ref{sec:consistency}). Finally, convergence is established in section \ref{sec:convergence}.

    \subsection{Setting} \label{sec:setting}

      We will consider the Vlasov--Poisson equations in 1+1 dimensions, i.e.
      \begin{equation} \label{eq:vp11}
	      \left\{
	      \begin{aligned}
		      \partial_t f(t,x,v) &= -v\partial_x f(t,x,v) - E(f(t,\cdot,\cdot),x) \partial_v f(t,x,v) \\
		      \partial_x E(f(t,\cdot,\cdot),x)   &= \int_{\mathbb{R}} f(t,x,v)\,\mathrm{d}v-1 \\
		      f(0,x,v)			&= f_0(x,v)
	      \end{aligned}
	      \right.
      \end{equation}
      with periodic boundary conditions in space. The domain of interest is given by $(t,x,v)\in [0,T]\times [0,L] \times \mathbb{R}$. The periodic boundary conditions imply
      \begin{eqnarray*}
	      \forall x\in\mathbb{R} \colon f(t,x,v)=f(t,x+L,v).
      \end{eqnarray*}
      It is physically reasonable to assume at least an algebraic decay of $f_0$ in the velocity direction. Thus, we can approximate (to arbitrary precision) $f_0$ by an initial value with compact support. As will be apparent in the next section it is unnecessary to impose boundary conditions in the velocity direction for initial values with compact support. This is due to the fact that for such an initial value the solution will continue to have compact support for all finite time intervals $[0,T]$ (see Theorem \ref{thm:regularity}).

      For most of this presentation it will be more convenient to work with the following abstract initial value problem
      \begin{equation} \label{eq:abstract_ivp}
	\left\{
	\begin{aligned}
		\partial_{t}f(t)&=(A+B)f(t)  &  \\
		f(0)&=f_{0},   \\
	\end{aligned}
	\right.
      \end{equation}
      where we assume that $A$ is an (unbounded) linear operator. In addition, we assume that $B$ can be written in the form $Bf=B(f)f$, where $B(f)$ is an (unbounded) linear operator. For the Vlasov--Poisson equation in 1+1 dimensions the obvious choice is $Af = -v\partial_x f$ and $Bf=-E(f(t,\cdot,\cdot),x) \partial_v f$.

      In $1+1$ dimensions an explicit representation of the electric field is given by the following formula
      \begin{equation}	\label{eq:em_field}
	\begin{aligned}
		E(f(t,\cdot,\cdot),x) &= \int_0^L K(x,y) \left( \int_\mathbb{R} f(t,y,v)\mathrm{d}v - 1 \right) \mathrm{d}y, \\
		K(x,y) &= \begin{cases}
		  \frac{y}{L} &\quad 0 < y < x, \\
	      \frac{y}{L}-1 &\quad x < y < L,
	\end{cases}
	\end{aligned}
      \end{equation}
      where $E$ is chosen to have zero integral mean (electrostatic condition). This representation allows us to get a simple estimate of the electric field in terms of the probability density $f$. Note, however, that all the estimates employed in this paper could just as well be derived from potential theory. In fact, this approach is preferred as soon as one considers more than a single dimension in the space direction.

      \subsection{Definitions and notation}

      The purpose of this section is to introduce the notations and mathematical spaces necessary for giving existence, uniqueness, and regularity results as well as to conduct the estimates necessary for showing consistency and convergence.

      First we introduce some notations that will be employed in the subsequent analysis. Suppose that the differential equation $g^{\prime} = G(g)$ has (for a given initial value) a unique solution. In this case we denote the solution at time $t$ with initial value $g(t_0)=g_0$ by $g(t)=E_G(t-t_0,g_0)$. In addition, we will use $\Vert \cdot \Vert$ to denote the infinity norm and $\Vert \cdot \Vert_p$ to denote the $L^p$ norm on $[0,L]\times \mathbb{R}$.

      For estimating the errors in space and velocity we will use the Banach space $L^\infty([0,L]\times[-v_{\max},v_{\max}])$. Note that consistency bounds in the physically more reasonable $L^1$ norm are a direct consequence of the bounds we derive in the infinity norm. The situation is more involved in the case of stability (this is discussed in section~\ref{sec:convergence}).

      For our convergence analysis we need some regularity of the solution. To that end, we introduce the following spaces of continuously differentiable functions
      \begin{eqnarray*}
	      \mathcal{C}_{\mathrm{per,c}}^m
	      &:=& \left\{g \in \mathcal{C}^m(\mathbb{R}^2,\mathbb{R})  ,\, \forall x,v \colon (g(x+L,v)=g(x,v)) \land (\text{supp}\, g(x,\cdot) \text{ compact})  \right\}, \\
	      \mathcal{C}_{\mathrm{per}}^m
	      &:=& \left\{g \in \mathcal{C}^m(\mathbb{R},\mathbb{R})  ,\, \forall x \colon g(x+L)=g(x) \right\}.
      \end{eqnarray*}	
      Equipped with the norm of uniform convergence of all derivatives up to order $m$,
      $\mathcal{C}_{\text{per,c}}^m$ and $\mathcal{C}_{\text{per}}$ are Banach spaces.

      We also have to consider spaces that involve time. To that end let us define for any subspace $Z\subset C^m(\mathbb{R}^d,\mathbb{R})$ the space
      \begin{eqnarray*}
	      \mathcal{C}^m(0,T;Z)
	      &:=& \left\{ g \in \mathcal{C}^m([0,T],C^0) ,\,
	      ( g(t) \in Z ) \land
	      ( \sup_{t\in[0,T]}\left\Vert g(t) \right\Vert_{Z} < \infty )
	      \right\}.
      \end{eqnarray*}
      Below, we will either take the choice $Z=\mathcal{C}_{\text{per,c}}^m$ or $Z=\mathcal{C}_{\text{per}}^m$. It should be noted that functions in $\mathcal{C}^m(0,T;Z)$ possess spatial derivatives up to order $m$ that are uniformly bounded in $t\in[0,T]$.

    \subsection{Existence, uniqueness, and regularity} \label{sec:regularity}

    In this section we recall the existence, uniqueness, and regularity results of the Vlasov--Poisson
    equations in 1+1 dimensions. The theorem is stated with a slightly different notation in \cite{Besse:2008} and \cite{Besse:2005}.

    \begin{theorem} \label{thm:regularity}
        Assume that $f_0 \in \mathcal{C}_{\mathrm{per,c}}^m$ is non-negative, then $f \in \mathcal{C}^m(0,T;\mathcal{C}_{\mathrm{per,c}}^m)$ and $E(f(t,\cdot, \cdot),x)$ as a function of $(t,x)$ lies in $\mathcal{C}^m(0,T;\mathcal{C}_{\mathrm{per}}^m)$. In addition, we can find a number $Q(T)$ such that for all $t\in[0,T]$ and $x\in\mathbb{R}$ it holds that $\mathrm{supp}\, f(t,x,\cdot) \subset [-Q(T),Q(T)]$.
    \end{theorem}
    \begin{proof}
	    A proof can be found in \cite{glassey:1996}.
    \end{proof}

    We also need a regularity result for the electric field that does not directly result from a
    solution of the Vlasov--Poisson equations, but from some generic function $f$ (e.g., an $f$ computed from an
    application of a splitting operator to $f_0$).

    \begin{corollary} \label{col:regB}
	    For $f \in  \mathcal{C}_{\mathrm{per,c}}^m$ it holds that
	    $E(f,\cdot)\in \mathcal{C}_{\mathrm{per}}^m$.
    \end{corollary}
    \begin{proof}
	    The result follows from the proof of Theorem \ref{thm:regularity}.
	    In addition, in the 1+1 dimensional case it can also be followed from
	    the exact representation of the electromagnetic field that is given in equation \eqref{eq:em_field}.
    \end{proof}

    It should also be noted that due to the proof of Theorem \ref{thm:regularity}, the regularity results
    given can be extended to the differential equations generated by $B$ and $B(g)$ (for any sufficiently regular $g$). Thus,
    Theorem \ref{thm:regularity} remains valid if $E_B(t,f_0)$ or $\mathrm{e}^{tB(g)}f_0$ is substituted for $f(t)$.

    \subsection{Time discretization} \label{sec:time_discretization}

      We use Strang-splitting for the time discretization of \eqref{eq:abstract_ivp}. This results in the scheme
      \begin{subequations}\label{eq:time-discr}
      \begin{equation}
      	f_{k+1} = S_k f_{k},
      \end{equation}
      where $f_k$ is the numerical approximation to $f(t)$ at time $t=k\dt$ with step size $\dt$. The splitting operator $S_k$ is the  composition of three abstract operators
      \begin{equation}\label{eq:time-strang}
	    S_k = S^{(A)}S^{(B)}_kS^{(A)},
      \end{equation}
      where
      \begin{equation}
        S^{(A)} = \e^{\frac{\dt}{2} A},\qquad S^{(B)}_k = \e^{\dt B(f_{k+1/2})}
      \end{equation}
      with $f_{k+1/2} =  \e^{\frac{\dt}{2}B(f_k)} \e^{\frac{\dt}{2}A} f_k.$ The choice of $f_{k+1/2}$ is such as to retain second order in the non-linear case while still only advection problems have to be solved in the numerical approximation (for more details see e.g. \cite{time_analysis}). Note that since $\e^{\frac{\dt}{2}B(f_k)}$ can be represented by a translation in the velocity direction only (which has no effect on the computation of the electric field) we can use here
      \begin{equation} \label{eq:f_k+1/2}
       f_{k+1/2} = S^{(A)} f_k.
      \end{equation}
      \end{subequations}
      This is convenient as the computation of $S^{(A)}f_k$ incurs no performance overhead in the actual computation.

    \subsection{Space discretization}  \label{sec:space_discretization}

      We proceed in two steps. First, we introduce a cutoff in the velocity direction, i.e.~we fix $v_{\text{max}}$ and consider the problem on the domain $[0,L]\times[-v_{\text{max}}, v_{\text{max}}]$. Note that for an initial value with compact support with respect to velocity and a sufficiently large $v_{\text{max}}$ this is still exact.

      Second, we introduce a discontinuous Galerkin approximation in both the space and velocity direction. For simplicity, we consider a uniform rectangular grid. In this case, the cell boundaries are given by the following relations
      \begin{align*}
	x_i = i \dx_x, \qquad &0 \leq i \le N_x, \\
	v_j = j \dx_v - v_{\text{max}}, \qquad &0 \leq j \le N_v.
      \end{align*}
      Within each cell, i.e.~a square $R_{ij} = [i \dx_x,(i+1) \dx_x]\times[j \dx_v - v_{\text{max}}, (j+1) \dx_v - v_{\text{max}}]$, $0 \leq i < N_x$, $0 \leq j < N_v$, we perform an orthogonal projection with respect to the basis of Legendre polynomials of degree at most $\order$ in $x$ and $v$. To be more precise, suppose that $g\in L^2\left( [0,L]\times[-v_{\text{max}}, v_{\text{max}}]\right)$; then the operator $P$ is defined such that $Pg$ restricted to $R_{ij}$ for all $i,j$ is the (unique) polynomial that results from the projection of $g$ onto the $(\order+1)(\order+1)$ dimensional subspace generated by the (appropriately translated and scaled) Legendre polynomials up to degree $\order$. It is well known that this projection operator is given by
      \begin{subequations}\label{eq:projection_restricted}
      \begin{equation}
	    Pg \vert_{R_{ij}} = \sum_{k=0}^{\order} \sum_{m=0}^{\order} b_{km}^{ij}P_k^{(1)}(x) P_m^{(2)}(v)
      \end{equation}
      with coefficients
      \begin{equation}
	     b_{km}^{ij}=\frac{(2k+1)(2m+1)}{\dx_x \dx_v} \int_{R_{ij}} P_m^{(1)}(x) P_k^{(2)}(v)g(x,v) \,\mathrm{d}(x,v).
      \end{equation}
      \end{subequations}
      The translated and scaled Legendre polynomials are here defined as
      \[
        P_l^{(1)}(\xi) = p_l\left(\frac{2(\xi-x_i)}{\dx_x}-1\right), \qquad P_l^{(2)}(\xi) = p_l\left(\frac{2(\xi-v_j)}{\dx_v}-1\right),
      \]
      where $p_l$ denote the Legendre polynomials with the standard normalization, i.e.
      \[
	    \int_{-1}^{1} p_l(x) p_j(x) \mathrm{d}x = \frac{2}{2l+1} \delta_{lj}.
      \]

      It should be emphasized that the projection in a single cell is independent from the projection in any other cell. As this is not true for Hermite or spline interpolation it gives the discontinuous Galerkin scheme a computational advantage (see \cite{Mangeney:2002} and section \ref{sec:translation_projection}).

      Now we have to introduce an approximation to the abstract splitting operator \eqref{eq:time-strang} that takes the space discretization into account. We use the decomposition
      \begin{subequations}\label{eq:full-discr}
      \begin{equation}
	    \tilde{S}_{k} = \tilde{S}^{(A)}\tilde{S}^{(B)}_k\tilde{S}^{(A)},
      \end{equation}
      where
      \begin{equation}
	    \tilde{S}^{(A)} = PS^{(A)},\qquad
	    \tilde{S}^{(B)}_k = P \e^{\dt B(\tilde{f}_{k+1/2})}
      \end{equation}
      with
      \begin{equation}
      	\tilde{f}_{k+1/2} = \tilde{S}^{(A)}\tilde{f}_k.
      \end{equation}
      The fully discrete scheme then reads
      \begin{equation}
		\tilde{f}_{k+1} = \tilde{S}_k \tilde{f}_k, \qquad
      	\tilde{f}_0 = P f_0.
      \end{equation}
      \end{subequations}
      Note that $\tilde{f}_k$ represents the full approximation in time and space at time $t_k$.

   \subsection{Translation and projection} \label{sec:translation_projection}
      The principle algorithm has already been laid out in sections \ref{sec:time_discretization} and \ref{sec:space_discretization}. However, the description given so far is certainly not sufficient as the straightforward implementation (first computing an exact solution and then projection onto a finite dimensional subspace) is clearly not a viable option. Thus, the purpose of this section is to describe in more detail the computation of
      \begin{equation*}
	\tilde{S}^{(A)}f(x,v) = P \e^{\frac{\dt}{2} A}f(x,v) = P f\left( x-\frac{\dt}{2}v,v \right)
      \end{equation*}
      and
      \begin{equation*}
	\tilde{S}^{(B)}_k f(x,v) = P\e^{\dt B(\tilde f_{k+1/2})} f(x,v) = P f\left(v,x -\dt E(\tilde f_{k+1/2},x) \right).
      \end{equation*}

      Without loss of generality let us consider a translation of the form $f\left( x - \dt g(v),v \right)$. In addition, we fix the cell of interest as $[0,\dx]\times[0,\dx]$.
      Now we are primarily interested in an interval of length $\dx$ and thus define $P_l(x) = p_l(\frac{2x}{\dx}-1)$. Then we have
      \[
	\int_{0}^{\dx} P_l(x) P_j(x) \mathrm{d}x = \frac{\dx}{2l+1} \delta_{lj}.
      \]
      We have to first translate and then project a function that can be expanded as
      \[
	f(x,v) = \sum_{m=0}^M \sum_{n=0}^N b_{mn} P_m(x) P_n(v)
      \]
      onto the finite dimensional approximation space. Our goal is to compute the coefficients of $f(x-\dt g(v) ,v)$. These are given by
      \begin{align}
	a_{lj}
	&= \frac{(2l+1)(2j+1)}{\dx^2}\int_0^{\dv} \int_0^{\dx} P_{l}(x)P_{j}(v)f(x-\dt g(v),v)\, \mathrm{d}x\, \mathrm{d}v  \notag\\
	&= \frac{(2l+1)(2j+1)}{\dx} \sum_{m,n}b_{mn}\int_0^{\dv} P_{j}(v)P_{n}(v)\left(\frac{1}{\dx}\int_0^{\dx} P_{l}(x)P_{m}(x- \dt g(v)) \,\mathrm{d}x \right) \,\mathrm{d}v \notag \\
	&= \frac{(2l+1)(2j+1)}{\dv}\sum_{m,n}b_{mn}\int_0^{\dv} P_{j}(v)P_{n}(v)H_{lm}(g(v)\dt/\dx ) \,\mathrm{d}v, \label{eq:projected}
      \end{align}
      where
      \begin{equation*}
	H_{lm}(\delta) = \frac{1}{\dx}\int_0^h P_{l}(x)P_{m}(x- \delta \dx) \,\mathrm{d}x, \qquad
	\delta = \frac{g(v) \dt}{\dx}.
      \end{equation*}
      For a fixed $v$ the function $H_{lm}$ can be evaluated explicitly. This is done, up to order $3$, in \cite{Mangeney:2002}. We will instead use a Mathematica program which can generate a representation of $H_{lm}$ (up to arbitrary order) in {\tt C} code that can then be embedded in the {\tt C++} implementation. Note that it is sufficient to only evaluate $H_{lm}$ for $0 < \delta < 1$ as the negative values of $\delta$ follow by a symmetry argument and integer multiplies correspond to a shift of the cells only. Also, the computation of $H_{lm}(\delta)$ for $-1<\delta<1$ shows that only two terms from the sum in \eqref{eq:projected} do not vanish. That is, we need only the data from the same cell as well as a single neighboring cell (either the right or left neighbor) to compute an application of a splitting operator. This follows easily from the fact that the support of the Legendre basis functions are within a single cell only. More details are given in \cite{Mangeney:2002}.

      It remains to evaluate the remaining integral in equation \eqref{eq:projected}. Since $g(v)$ is at most a polynomial of degree $\order$ (in a single cell) we have to integrate a polynomial of degree at most $\order^2$. We use a Gauss--Legendre quadrature rule of appropriate order.

      Note that in order to guarantee the stability of our scheme it is of vital importance that we can compute the exact result of the integral in equation \eqref{eq:projected}. If only an approximation is used instabilities can occur (see section \ref{sec:numerical_stability} and \cite{morton:1988}).

    \subsection{Polynomial approximation of functions with a small jump discontinuity} \label{sec:jackson_for_discontinuous}

      In this section our goal is to prove a bound concerning the approximation of piecewise polynomials of degree $\order$ with a single jump discontinuity. For notational simplicity we will be concerned with a function of a single variable only; the general case is a simple tensor product of the situation described in this section. Thus, the operator $P$ is here understood as the orthogonal projection with respect to the one-dimensional Legendre polynomials of degree less or equal to $\order$.
      The starting point of our investigation is the result in Theorem \ref{thm:jackson}, which is applicable only if we can assume that $g$ is $\order+1$ times continuously differentiable. This assumption is not satisfied for the discontinuous Galerkin approximation considered in this paper. However, we will use the result as a stepping stone to prove a similar bound for the approximation of functions with a small jump discontinuity.
      \begin{theorem} \label{thm:jackson} Suppose that $g\in\mathcal{C}^{\order +1}([0,\dx])$. Then
	\[
	  \left\Vert g^{(k)} - (Pg)^{(k)} \right\Vert \leq C  \dx^{\order-k+1} \Vert g^{(\order+1)}\Vert
	\]
      for all $k\in\{0\ldots,\order\}$.
      \end{theorem}
      \begin{proof}
      		In \cite[p.~59]{phillips03} it is shown that $Pg-g$ changes sign $\order+1$ times. From this, it follows that $(Pg)^{(k)}-g^{(k)}$ changes sign $\order+1-k$ times. Therefore, $(Pg)^{(k)}$ is an interpolation polynomial of $g^{(k)}$ of degree $\order+1-k$. Using the standard error representation for polynomial interpolation we get the desired result.\end{proof}

      For numerical methods that rely on a smooth approximation of the solution (for example, using Hermite or spline interpolation as in \cite{BesseMehr:2008}) sufficient regularity in the initial condition implies the bound given in Theorem \ref{thm:jackson} for any approximation that has to be made in the course of the algorithm.

      This assumption, however, is violated if we consider a discontinuous Galerkin approximation as, even if the initial condition is sufficiently smooth, the approximation will include a jump discontinuity at the cell boundary. Thus, we are interested in a bound that still gives us an equivalent result to that stated in Theorem \ref{thm:jackson} in the case of a function with a small jump discontinuity. The following theorem is thus the central result of this section. For simplicity, we consider a single cell only.
      \begin{theorem} \label{thm:discont_approx} Suppose that $g\colon [0,\dx]\to\mathbb{R}$ is piecewise polynomial of degree $\order$ with a single discontinuity at $x_0\in[0,\dx]$. In addition, we assume that the jump heights $\varepsilon^{(k)} =  g^{(k)}(x_0+)-g^{(k)}(x_0-) $ satisfy $\vert \varepsilon^{(k)} \vert \leq c \dx^{\order-k+1}$ for all $k \in \{0,\dots,\order\}$. Then,
      \[
	\left\Vert g^{(k)} - (Pg)^{(k)} \right\Vert \leq C \dx^{\order - k + 1},
      \]
      for all $k \in \{0,\dots,\order\}$. Note that the constant $C$ only depends on $c$ and the constant in Theorem~\ref{thm:jackson}.
      \end{theorem}
      \begin{proof}
	Let us assume that $x_0\in(0,\dx)$ (otherwise the result is immediate). We smooth the piecewise constant function $g^{(\order)}$ in the following way
	\[
	  p^{(\order)}(x) = \frac{\varepsilon^{(\order)}}{\dx}x+g^{(\order)}(0).
	\]
	Now, upon integration we get
	\[
	  p(x) = \frac{\varepsilon^{(\order)}}{\dx}\frac{x^{\order+1}}{(\order+1)!}+\sum_{k=0}^{\order} a_{k}x^{k},
	\]
	where we choose the coefficients in such a way that the Taylor polynomial of $g$ expanded at $0$ matches the first $\order$ terms of $g$, i.e. $a_{k}=\frac{g^{(k)}(0)}{k!}$. This gives us the following representation
	\begin{equation} \label{eq:p}
	  p(x) = \frac{\varepsilon^{(\order)}}{\dx}\frac{x^{\order+1}}{(\order+1)!}+\sum_{k=0}^{\order}\frac{g^{(k)}(0)}{k!}x^{k}.
	\end{equation}
	Now let us consider the integral (for $x>x_0$)
	\begin{align*}
	  \int_{0}^{x}p^{(m)}(y)-g^{(m)}(y)\,\mathrm{d}y
	  &= p^{(m-1)}(x)-g^{(m-1)}(x)-p^{(m-1)}(0)+g^{(m-1)}(0) \\ &\qquad + g^{(m-1)}(x_0+) - g^{(m-1)}(x_0-) \\
	  &= p^{(m-1)}(x)-g^{(m-1)}(x) + \varepsilon^{(m-1)},
	\end{align*}
	where the last identity follows from the choice we made above. Now we know that (for $s_{\order-1} > x_0$)
	\[
	  \int_{0}^{s_{\order-1}} p^{(\order)}(s_{\order})-g^{(\order)}(s_{\order})\, \mathrm{d}s_{\order}
	  =  p^{(\order-1)}(s_{\order-1})-g^{(\order-1)}(s_{\order-1}) + \varepsilon^{(\order-1)}  \\
	\]
	and further (for $s_{\order-2} > x_0$)
	\begin{align*}
	  \int_{0}^{s_{\order-2}} p^{(\order-1)}(s_{\order-1})-g^{(\order-1)}(s_{\order-1}) + \varepsilon^{(\order-1)}   \, \mathrm{d}s_{\order-1}
	  &= p^{(\order-2)}(s_{\order-2})-g^{(\order-2)}(s_{\order-2}) \\ &\qquad + \varepsilon^{(\order-2)} + \varepsilon^{(\order-1)} s_{\order-2} .
	\end{align*}
	By an induction argument we can then estimate the approximation error as
	\begin{align*}
	  \vert p(x)-g(x) \vert
	  &\leq \left| \int_{0}^{x} \int_0^{s_1} \dots\int_{0}^{s_{\order-1}}p^{(\order)}(s_{\order})-g^{(\order)}(s_{\order})\,
    \mathrm{d}s_{\order} \dots \mathrm{d}s_2 \mathrm{d}s_1 \right| + \sum_{k=0}^{\order-1} \vert \varepsilon^{(k)}\vert  h^{k}  \\
	  &\leq  \int_{0}^{x} \int_0^{s_1} \dots\int_{0}^{s_{\order-1}} \left| p^{(\order)}(s_{\order})-g^{(\order)}(s_{\order}) \right| \,
    \mathrm{d}s_{\order} \dots \mathrm{d}s_2 \mathrm{d}s_1  + c \dx^{\order+1} \\
	  &\leq c \dx^{\order+1}.
	\end{align*}
	In addition we easily follow from equation \eqref{eq:p} that
	\[
	  \Vert p^{(\order+1)} \Vert \leq \frac{\vert\varepsilon^{(\order)}\vert}{\dx} \leq c.
	\]
	Now let us estimate the approximation error
      \begin{align*}
	\Vert Pg-g \Vert
	&=\Vert Pg-Pp+Pp-p+p-g\Vert \\
	&\leq \Vert P(g-p)\Vert+\Vert Pp-p\Vert+\Vert p-g\Vert \\
	&\leq C \dx^{\order+1},
      \end{align*}
      where in the last line we have used Theorem \ref{thm:jackson} and the well known fact that the projection operator $P$ is a bounded operator in the infinity norm. The latter can be seen, for example, by estimating \eqref{eq:projection_restricted}.

      To get the corresponding result for the $k$th derivative we follow largely the same argument. The last estimate is then given by
    \begin{align*}
	\left\Vert (Pg)^{(k)}-g^{(k)} \right\Vert
	&\leq \left\Vert \left(P(g-p)\right)^{(k)} \right\Vert + \left\Vert p^{(k)}-g^{(k)} \right\Vert + \left\Vert (Pp)^{(k)}-p^{(k)} \right\Vert \\
	&\leq Ch^{-k}\left\Vert P(g-p)\right\Vert +C \dx^{\order-k+1} \left\Vert p^{(\order+1)}\right\Vert \\
	&\leq C \dx^{\order-k+1},
    \end{align*}
    where the estimate for the first term follows by the well-known Markov inequality (see e.g. \cite{shadrin2004twelve}).
    \end{proof}
       
      Let us discuss the principle of applying Theorem \ref{thm:discont_approx}. First the operator $P$ is applied to $f(j\dt)$, i.e.~a point on the exact solution, and we can assume the necessary regularity to apply Theorem \ref{thm:jackson}. Consequently, we get a jump discontinuity of heights at most
      \[
	\vert \varepsilon^{(k)} \vert \leq 2 \left\Vert f^{(k)}(j\dt)-Pf^{(k)}(j\dt)\right\Vert \leq C\dx^{\order-k+1}\left\Vert f^{(\order+1)}(j\dt)\right\Vert \leq C \dx^{\order-k+1}, \quad 0\leq k \leq \order.
      \]
      Now the projected function is translated by a splitting operator (the example $g(x)=(Pf(j\tau))(x-v\tau)$ is illustrated in Figure \ref{fig:discontinuous_projection}) and projected back on the finite dimensional subspace. The resulting error up to the $\order$-th derivative is then given by (see Theorem \ref{thm:discont_approx})
       \[
	\left\Vert g^{(k)} - (Pg)^{(k)} \right\Vert \leq C \dx^{\order - k + 1}.
      \]

      From this we can also follow that the new jump heights $\varepsilon_1^{(k)}$ are at most
      \[
	  \vert \varepsilon_1^{(k)} \vert \leq 2 \Vert (Pg)^{(k)}-g^{(k)}\Vert \leq C \dx^{\order-k+1}, \qquad 0 \leq k \leq \order.
      \]

       \begin{figure}[h]
  	\centering \includegraphics[width=6cm]{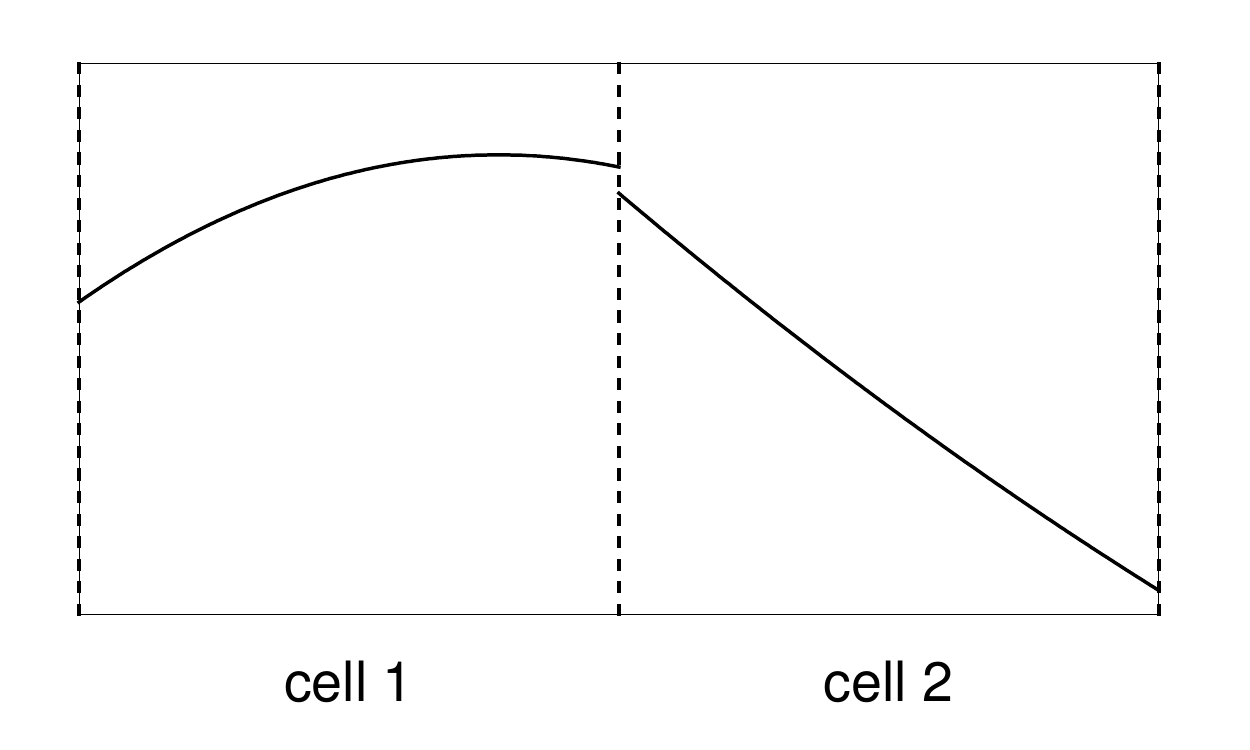}
 	\centering \includegraphics[width=6cm]{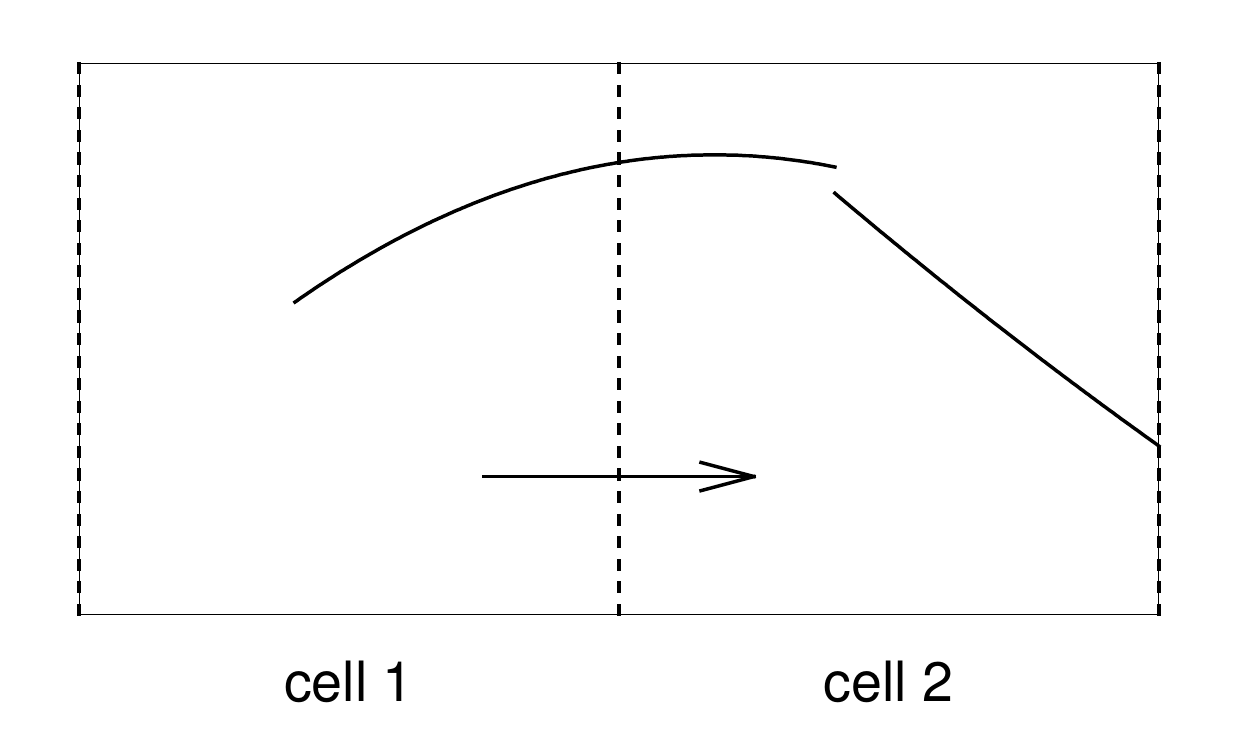}
 	\caption{Projected smooth function with a jump discontinuity at the cell boundary (left) and translation with a discontinuity inside the cell (right). Only two cells in the x direction are shown. \label{fig:discontinuous_projection}}
       \end{figure}

      Since we only have to repeat this procedure a finite number of times (i.e.~for a single step of the Strang splitting algorithm) and the assumptions of Theorem \ref{thm:discont_approx} are satisfied uniformly for all $f(t)$, we can find a uniform constant $C$ such that the desired estimate holds.

      Strictly speaking this argument is only valid for a constant advection (i.e. where $v$ is fixed). However, we can always decompose the projection operator as $P=P_v P_x$; that is, into a projection in the $x$-direction (that depends on the parameter $v$) and a subsequent projection in the $v$-direction. Due to the special form of the advection we consider (see section \ref{sec:translation_projection}), the projection in the $x$-direction gives   a function that is piecewise polynomial in every cell. Thus, the projection in the $v$-direction poses no difficulty.

    \subsection{Consistency} \label{sec:consistency}

      It is the purpose of this section to formulate assumptions under which we can show a consistency bound for the initial value problem given in equation \eqref{eq:vp11}. For notational simplicity, we will denote in this section the solution of \eqref{eq:vp11} at a fixed time $t_k=k \tau$ by $f_0$. The function $\tilde{f}_0$ defined as $P f_0$ is a (possible) initial value for a single time step (i.e., a single application of the splitting operator $S_k$ or $\tilde{S}_k$). Since we consider consistency we are interested in the non-linear operator $S$ that is given by
	  \[
	    S(\cdot) = S^{(A)} \e^{\dt B\left(S^{(A)}(\cdot)\right)} S^{(A)} (\cdot),
	  \]
	  and the corresponding spatially discretized operator
      \[
        \tilde{S}(\cdot) = \tilde{S}^{(A)} P\e^{\dt B\left(\tilde{S}^{(A)}(\cdot)\right)} \tilde{S}^{(A)} (\cdot).
	  \]

      Let us first give a simple consequence of the variation-of-constants formula.
      \begin{lemma} \label{lem:B_estimate} Suppose that $f_0$ is absolutely continuous and that $f_{1/2}, \tilde{f}_{1/2}$ are integrable. Then
	\[
	S^{(B(f_{1/2}))}f_0-S^{(B(\tilde{f}_{1/2}))}f_0 =
    \int_{0}^{\dt} \e^{(\dt-\sigma)B(\tilde{f}_{1/2})}\left(B(f_{1/2})-B(\tilde{f}_{1/2})\right) \e^{\sigma B(f_{1/2})}f_0 \, \mathrm{d}\sigma.
	\]
      \end{lemma}
      \begin{proof}
	Let $g(\dt)=S^{(B(f_{1/2}))}f_{0}$. Then
	\[
	  g^{\prime}=B(f_{1/2})g=B(\tilde{f}_{1/2})g+\left(B(f_{1/2})-B(\tilde{f}_{1/2})\right)g,
	\]
	which can be rewritten by the variation-of-constants formulas as
	\[
    S^{(B(f_{1/2}))}f_{0}=S^{(B(\tilde{f}_{1/2}))}f_{0}+\int_{0}^{\dt}\e^{(\dt-\sigma)B(\tilde{f}_{1/2})} \left(B(f_{1/2})-B(\tilde{f}_{1/2})\right)g(\sigma)\,\mathrm{d}\sigma,
	\]
	from which the desired result follows immediately.
      \end{proof}

      The next two lemmas will be the crucial step to prove consistency. First, we consider the error made by the (exact) splitting operators due to the space discretization. Note that the assumptions are exactly the same as those needed for the development in section \ref{sec:jackson_for_discontinuous} to hold.

      \begin{lemma} \label{lem:consistency_1} Suppose that $f_0\in\mathcal{C}_{\mathrm{per,c}}^{\order+1}$. Then
	\[
	  \Vert Sf_{0}-S\tilde{f}_{0}\Vert \leq C \left( \dt \dx^{\order+1}+ \dx^{\order+1} \right),
	\]
	where $C$ depends on $\Vert f_0 \Vert_{\mathcal{C}_{\mathrm{per,c}}^{\order+1}}$ (but not on $\tau$ and $h$).
      \end{lemma}
      \begin{proof}
	Let us define $\hat{f}_{1/2} = S^{(A)}\tilde{f}_0$. Then, we can write
	\begin{align*}
	  \Vert Sf_{0}-S\tilde{f}_{0}\Vert
	  &=  \left\Vert S^{(A)}S^{(B(f_{1/2}))}S^{(A)}f_{0}-S^{(A)}S^{(B(\hat{f}_{1/2}))}S^{(A)}\tilde{f}_{0} \right\Vert \\
	  &\leq \left\Vert S^{(A)}\left(S^{(B(f_{1/2}))}-S^{(B(\hat{f}_{1/2}))}\right)S^{(A)}f_{0} \right\Vert + \left\Vert S^{(A)}S^{(B(\hat{f}_{1/2}))}S^{(A)}(f_{0}-\tilde{f}_{0}) \right\Vert \\
	  &\leq  \left\Vert S^{(A)}\left(S^{(B(f_{1/2}))}-S^{(B(\hat{f}_{1/2}))}\right)S^{(A)}f_{0} \right\Vert + C \Vert f_{0}-\tilde{f}_{0} \Vert.
	\end{align*}
	By using Lemma \ref{lem:B_estimate} and the definition of $B$ we get
	\begin{align*}
	  & \left\Vert S^{(A)}\left(S^{(B(f_{1/2}))}-S^{(B(\hat{f}_{1/2}))}\right)S^{(A)}f_{0} \right\Vert \\
	  &\qquad\qquad \leq C\dt\Vert E(f_{1/2})-E(\hat{f}_{1/2})\Vert\max_{\sigma\in[0,\dt]}\left\Vert\partial_{v}\left(\e^{\sigma B(f_{1/2})}S^{(A)}f_{0}\right)\right\Vert_{L^{\infty}}.
	\end{align*}
	Note that the set of measure zero, where $\partial_v$ is not defined, does not influence the estimate in the $L^\infty$ norm as we are only concerned with equivalence classes of (essentially) bounded functions.
	Finally, since $E$ is given by equation \eqref{eq:em_field} it follows that
	\[
	  \Vert E(f_{1/2})-E(\hat{f}_{1/2})\Vert \leq \Vert f_{1/2}-\hat{f}_{1/2} \Vert \leq \Vert f_0 - \tilde{f}_0 \Vert \leq C \dx ^{\order+1},
	\]
	which concludes the proof.
      \end{proof}

      Second, we consider the error made due to the approximation of the (exact) splitting operators. Note that the assumptions are exactly the same as those needed for the development in section \ref{sec:jackson_for_discontinuous}.

      \begin{lemma} \label{lem:consistency_2} Suppose that $f_0\in\mathcal{C}_{\mathrm{per,c}}^{\order+1}$. Then
	\[
	  \Vert S\tilde{f}_{0}-\tilde{S}\tilde{f}_{0}\Vert \leq C \left( \dt \dx^{\order+1}+ \dx^{\order+1} \right),
	\]
	where $C$ depends on $\Vert f_0 \Vert_{\mathcal{C}_{\mathrm{per,c}}^{\order+1}}$ (but not on $\tau$ and $h$).
      \end{lemma}
      \begin{proof}
	Let $\hat{f}_{1/2} = S^{(A)}\tilde{f}_0$. Then, we can write
	\begin{align*}
	  S\tilde{f}_{0}-\tilde{S}\tilde{f}_{0}
	  &= S^{(A)}S^{(B(\hat{f}_{1/2}))}S^{(A)}\tilde{f}_{0}-\tilde{S}^{(A)} \tilde{S}^{\left(B(\tilde{f}_{1/2})\right)}\tilde{S}^{(A)}\tilde{f}_{0} \\
	  &= \left(S^{(A)}-\tilde{S}^{(A)}\right)S^{(B(\hat{f}_{1/2}))}S^{(A)}\tilde{f}_{0} \\
	  &\qquad + \tilde{S}^{(A)}\left(S^{(B(\hat{f}_{1/2}))}-\tilde{S}^{(B(\tilde{f}_{1/2}))}\right) S^{(A)}\tilde{f}_{0} \\
	  &\qquad + \tilde{S}^{(A)}\tilde{S}^{(B(\tilde{f}_{1/2}))}\left(S^{(A)}-\tilde{S}^{(A)}\right)\tilde{f}_{0}.
	\end{align*}
	Now we estimate the three terms independently. The estimation of the first and third term is straightforward. We get
	\begin{align*}
	    \left\Vert \left(S^{(A)}-\tilde{S}^{(A)}\right)S^{(B(\hat{f}_{1/2}))}S^{(A)}\tilde{f}_{0} \right\Vert
	    &= \left\Vert (P-1)\left(S^{(A)}S^{(B(\hat{f}_{1/2}))}S^{(A)}\tilde{f}_{0}\right) \right\Vert \\
	    &\leq C\dx^{\order+1}
	\end{align*}
	and
	\begin{align*}
	    \left\Vert \tilde{S}^{(A)}\tilde{S}^{(B(\tilde{f}_{1/2}))}\left(S^{(A)}-\tilde{S}^{(A)}\right)\tilde{f}_{0} \right\Vert
	    &\leq C \left\Vert \left(S^{(A)}-\tilde{S}^{(A)}\right)\tilde{f}_{0} \right\Vert \\
	    &= C \left\Vert (1-P)S^{(A)}\tilde{f}_{0} \right\Vert \\
	    &\leq C \dx^{\order+1}.
	\end{align*}
	To estimate the second term we employ Lemma~\ref{lem:B_estimate} which gives
	\begin{align*}
	  & \left\Vert \tilde{S}^{(A)}\left(S^{(B(\hat{f}_{1/2}))}-\tilde{S}^{(B(\tilde{f}_{1/2}))}\right)S^{(A)}\tilde{f}_{0} \right\Vert \\
	  &\leq C\tau\Vert E(\hat{f}_{1/2})-E(\tilde{f}_{1/2})\Vert\max_{\sigma\in[0,\dt]}\left\Vert\partial_{v}\left(\e^{\sigma B(\hat{f}_{1/2})}S^{(A)}\tilde{f}_{0}\right)\right\Vert_{L^{\infty}}.
	\end{align*}
	Note that the set of measure zero, where $\partial_v$ is not defined, does not influence the estimate in the $L^\infty$ norm as we are only concerned with equivalence classes of (essentially) bounded functions. As in the last lemma $E$ is given by equation \eqref{eq:em_field} and thus it follows that
	\[
	  \Vert E(\hat{f}_{1/2})-E(\tilde{f}_{1/2})\Vert \leq \Vert \hat{f}_{1/2}-\tilde{f}_{1/2} \Vert = \Vert (1-P)\hat{f}_{1/2} \Vert \leq C \dx^{\order+1},
	\]
	which concludes the proof.
      \end{proof}

      \begin{theorem} [{\rm Consistency}] \label{thm:consistency}  Suppose that $f_0\in\mathcal{C}_{\mathrm{per,c}}^{\mathrm{max}(\order+1,3)}$. Then
	\[
	    \Vert Pf(h)-\tilde{S}\tilde{f}_{0} \Vert \leq C \left( \dt^3 + \dt \dx^{\order+1} + \dx^{\order+1} \right),
	\]
	where $C$ depends on $\Vert f_0 \Vert_{\mathcal{C}_{\mathrm{per,c}}^{\mathrm{max}(\order+1,3)}}$ (but not on $\tau$ and $h$).
      \end{theorem}
      \begin{proof}
	We write
	\begin{align*}
	  \Vert Pf(h)-\tilde{S}\tilde{f}_{0}\Vert
	  &=\Vert Pf(h)-PSf_{0}+PSf_{0}-Sf_{0}+Sf_{0}-S\tilde{f}_{0}+S\tilde{f}_{0}-\tilde{S}\tilde{f}_{0}\Vert \\
	  &\leq   \Vert P\left(f(h)-Sf_{0}\right)\Vert+\Vert PSf_{0}-Sf_{0}\Vert+\Vert Sf_{0}-S\tilde{f}_{0}\Vert+\Vert S\tilde{f}_{0}-\tilde{S}\tilde{f}_{0}\Vert \\
	  &\leq C\dt^{3}+C\dx^{\order+1}+\Vert Sf_{0}-S\tilde{f}_{0}\Vert+\Vert S\tilde{f}_{0}-\tilde{S}\tilde{f}_{0}\Vert,
	\end{align*}
    where the first term was bounded by Theorem 4.9 in \cite{time_analysis}. The two remaining terms can be bounded by Lemmas \ref{lem:consistency_1} and \ref{lem:consistency_2} to give the desired estimate.
      \end{proof}

    \subsection{Convergence} \label{sec:convergence}

    To show consistency it was most convenient to bound all terms in the infinity norm. Bounds in the $L^{1}$ or $L^{2}$ norms then follow since we consider a compact domain in space and velocity. However, for stability (and thus convergence) we need to bound the operator norm of the projection operator $P$ by $1$. Since such a bound is readily available in the $L^2$ norm  (as an orthogonal projection is always non-expansive in the corresponding norm) we will use it to show convergence. Note that this is not a peculiarity of our discontinuous Galerkin scheme. For example, in \cite{BesseMehr:2008} stability for two schemes based respectively on spline and Hermite interpolation is shown in the $L^2$ norm only.

    \begin{theorem} [{\rm Convergence}] \label{thm:convergence} For the numerical solution of \eqref{eq:vp11} we employ the scheme \eqref{eq:full-discr}. Suppose that the initial value $f_0\in\mathcal{C}^{\max\{\order+1,3\}}$ is non-negative and compactly supported in velocity. Then, the global error satisfies the bound
    \[
      \sup_{0 \leq n \leq N}  \left\Vert \left(\prod_{k=0}^{n-1}\tilde{S}_{k}\right)\tilde{f}_{0}-f(n\dt) \right\Vert_2
      \leq  C \left( \dt^2 + \frac{\dx^{\order+1}}{\dt} + \dx^{\order+1} \right),
    \]
    where $C$ depends on $T$ but is independent of $\tau,h,n$ for $0 \leq n \tau \leq N \tau = T$.
    \end{theorem}
    \begin{proof}
		From \eqref{eq:full-discr} we get
		\[
		  \tilde{f}_{n+1/2}=\tilde{S}^{(A)}\left(\prod_{m=0}^{n-1}\tilde{S}^{(A)}P\e^{\tau B(\tilde{f}_{m+1/2})}\tilde{S}^{(A)}\right)\tilde{f}_{0}.
		\]
		Now we can derive a recursion for the error in the $L^2$ norm
		\begin{align*}
			e_{n+1}&=\Vert\tilde{f}_{n+1}-f(n\tau+\tau)\Vert_{2}\\
            &=\Vert\tilde{S}^{(A)}P\e^{\tau B(\tilde{f}_{n+1/2})}\tilde{S}^{(A)}\tilde{f}_{n}-f(n\tau+\tau)\Vert_{2}\\
            &\leq\Vert\tilde{S}^{(A)}P\e^{\tau B(\tilde{f}_{n+1/2})}\tilde{S}^{(A)}\tilde{f}_{n}-\tilde{S}^{(A)}P\e^{\tau B(P\e^{\frac{\tau}{2}A}Pf(n\tau))}\tilde{S}^{(A)}\tilde{f}_{n}\Vert_{2}\\
            &\qquad +\Vert\tilde{S}^{(A)}P\e^{\tau B(P\e^{\frac{\tau}{2}A}Pf(n\tau))}\tilde{S}^{(A)} ( \tilde{f}_{n}-f(n\tau)) \Vert_{2}\\
            &\qquad+\Vert\tilde{S}^{(A)}P\e^{\tau B(P\e^{\frac{\tau}{2}A}Pf(n\tau))}\tilde{S}^{(A)}(1-P)f(n\tau)\Vert_{2}\\
            &\qquad+\Vert\tilde{S}^{(A)}P\e^{\tau B(P\e^{\frac{\tau}{2}A}Pf(n\tau))}\tilde{S}^{(A)}Pf(n\tau)-Pf(n\tau+\tau)\Vert_{2}\\
            &\qquad+\Vert(P-1)f(n\tau+\tau)\Vert_{2}.
        \end{align*}
        The first term is estimated with the help of Lemma~\ref{lem:B_estimate}, the fourth one with Theorem~\ref{thm:consistency}. This gives us
        \begin{align*}
            e_{n+1}&\leq C\tau\Vert\tilde{f}_{n+1/2}-P\e^{\frac{\tau}{2}A}Pf(n\tau)\Vert_{2}+ \Vert\tilde{f}_{n}-f(n\tau)\Vert_{2}+C\left(\tau^{3}+\tau h^{\ell+1}+h^{\ell+1}\right)\\
            &\leq (1+C\tau)e_{k}+C\left(\tau^{3}+\tau h^{\ell+1}+h^{\ell+1}\right).
		\end{align*}
		Applying a discrete Gronwall lemma to the above recursion then gives
		\begin{align*}
			e_{n+1}&\leq \e^{CT}e_{0}+C\left(\tau^{2}+h^{\ell+1}+\frac{h^{\ell+1}}{\tau}\right)\\&\leq C\left(\tau^{2}+h^{\ell+1}+\frac{h^{\ell+1}}{\tau}\right),
		\end{align*}
		which is the desired bound as the constant $C$ can be chosen uniformly in $[0,T]$. This follows from the regularity result (Theorem \ref{thm:regularity}) which gives us the desired bound for Theorem \ref{thm:consistency} if $f_0\in\mathcal{C}^{\max\{\order+1,3\}}$ is non-negative and compactly supported with respect to velocity.
      \end{proof}
  \subsection{Extension to higher dimensions}

    In three dimensions the splitting scheme is given by (for simplicity we consider a single time step only and thus drop the corresponding indices)
    \begin{align}
      S^{(A)}f(\boldsymbol{x},\boldsymbol{v}) &= f\left(\boldsymbol{x}-\frac{\dt}{2}\boldsymbol{v},\boldsymbol{v}\right), \label{eq:SA_3d} \\
      S^{(B)}f(\boldsymbol{x},\boldsymbol{v}) &= f(\boldsymbol{v},\boldsymbol{x}-\dt \boldsymbol{E}(f_{1/2},\boldsymbol{x})). \label{eq:SB_3d}
    \end{align}
    The expression in equation \eqref{eq:SA_3d} can be easily decomposed into three translation in a single dimension, i.e.
    \[
      S^{(A)} = \e^{\frac{\dt}{2} A_x} \e^{\frac{\dt}{2} A_y} \e^{\frac{\dt}{2} A_z}
    \]
    with $A_x = -v_x \partial_x$, $A_y = -v_y \partial_y$, and $A_z = -v_z \partial_z$.

    The discussion is more subtle for the expression in equation \eqref{eq:SB_3d}. In this case we can still use the decomposition given above; however, if we introduce a space discretization we will have to project back not onto a $1+1$ dimensional space but onto $1+3$ dimensional space. This is an important implementation detail; however, the convergence proof is (except for notational difficulties) unaffected.

    As most of the derivation in this paper is conducted within the framework of abstract operators, the extension to multiple dimensions is straightforward. In Lemmas \ref{lem:consistency_1} and \ref{lem:consistency_2} we have to consider a more general differentiation operator (i.e.\ a directional derivative). However, this represents no difficulty as the existence and regularity results are not restricted to the $1+1$ dimensional case (see \cite{glassey:1996}).

    Therefore, it remains to generalize the discussion given in \cite{time_analysis} to multiple dimensions in the case of the Vlasov--Poisson equation. The abstract results hold independently of the dimension and the specific details of the operators $A$ and $B$. The remaining computations are somewhat tedious, however, as the existence and regularity results are essentially the same the proof can be extended in a straightforward fashion.

  \section{Numerical simulations} \label{sec:numerical_simulation}

    The purpose of this section is to perform a number of numerical simulations in order to establish the validity of the implementation. The recurrence phenomenon in the context of higher order implementations in space is discussed in section \ref{sec:recurrence}. In section \ref{sec:order} the order of the method in the strong Landau damping problem is investigated. We will also reproduce some medium time integration results for linear Landau damping (section \ref{sec:landau_damping}) and investigate the stability for the Molenkamp--Crowley test problem (section \ref{sec:numerical_stability}).

    The computer program used for the numerical simulations performed in this section is implemented in {\tt C++}. It employs heavily the concept of templates and operator overloading to provide a compact implementation that is easily extendable to the multi dimensional case. As a result, the core program consists of only about 800 lines of source code (excluding unit tests but including all the logic needed to carry out the simulations in this section) while still maintaining an implementation with good performance.

    \subsection{Recurrence} \label{sec:recurrence}
      It is well known that piecewise constant approximations in velocity space lead to a recurrence phenomenon that is purely numerical in origin. This behavior has been investigated, for example, in \cite{Cheng:1976} and \cite{jenab:2011}. In \cite{zhou:2001} it is demonstrated by a number of numerical experiments that in the weak Landau damping problem this phenomenon is also purely a numerical artefact.

      From an analytical point of view the recurrence phenomenon is most easily understood for an advection equation, i.e.~a function $f(t,x,v)$ satisfying the equation
      \begin{equation}
	\partial_t f = -v \partial_x f. \label{eq:advection}
      \end{equation}
      For its numerical solution consider a piecewise constant approximation of $f$ in velocity space. This approximation results in slices in velocity space that correspond to the average velocity in a particular cell. Let us further assume that the velocity space $[-v_{\max},v_{\max}]$ consists of an odd number of cells and that the interval $[0,4 \pi]$ is employed in the space direction. Then the solution of \eqref{eq:advection} is a periodic function in time and the period $p$ is easily determined to be
      \[
	\dx_v p = 4 \pi.
      \]
      That this is only a numerical artefact is a simple consequence of the fact that $p$ tends to infinity as $\dx_v$ tends to $0$. However, for the purpose of this section it is instructive to compute the exact solution for the following initial value
      \[
	f_0(x,v) = \frac{\e^{-v^2/2}}{\sqrt{2 \pi}}\big(1+0.01 \cos(0.5 x)\big).
      \]
      The solution of \eqref{eq:advection} is then given by
      \[
	f(t,x,v) = \frac{\e^{-v^2/2}}{\sqrt{2 \pi}}\big(1+0.01 \cos(0.5 x - 0.5 v t)\big).
      \]
      This function, however, is not periodic in time (with a period being independent of $v$). To represent this more clearly, we compute the electric energy
      \begin{equation} \label{eq:electric_energy_advection}
	\mathcal{E}(t) = \int_0^{4 \pi} E(t,x)^2 \mathrm{d}x = \frac{\pi}{1250} \e^{-0.25 t^2},
      \end{equation}
      where the electric field $E(t,x)$ is determined as before by
      \[
      	E(t,x) = \int_0^L K(x,y) \left( \int_\mathbb{R} f(t,y,v)\mathrm{d}v - 1 \right) \mathrm{d}y.
      \]
      Note that the kernel $K(x,y)$ is defined in equation \eqref{eq:em_field}.
      Thus, the electric energy is exponentially decreasing for the exact solution (but periodic in time for the numerical solution). One might naively expect that this phenomenon vanishes as soon as one considers an approximation of degree at least $1$ in the velocity direction. While it is true that the solution is no longer periodic, as can be seen from Figure \ref{fig:recurrence_advection}, errors in the velocity approximation still result in a \emph{damped recurrence} of the electric field. Note that the size of this recurrence effect seems to be determined by the space discretization error.
      \begin{figure}[h]
 	\centering \includegraphics[width=12cm]{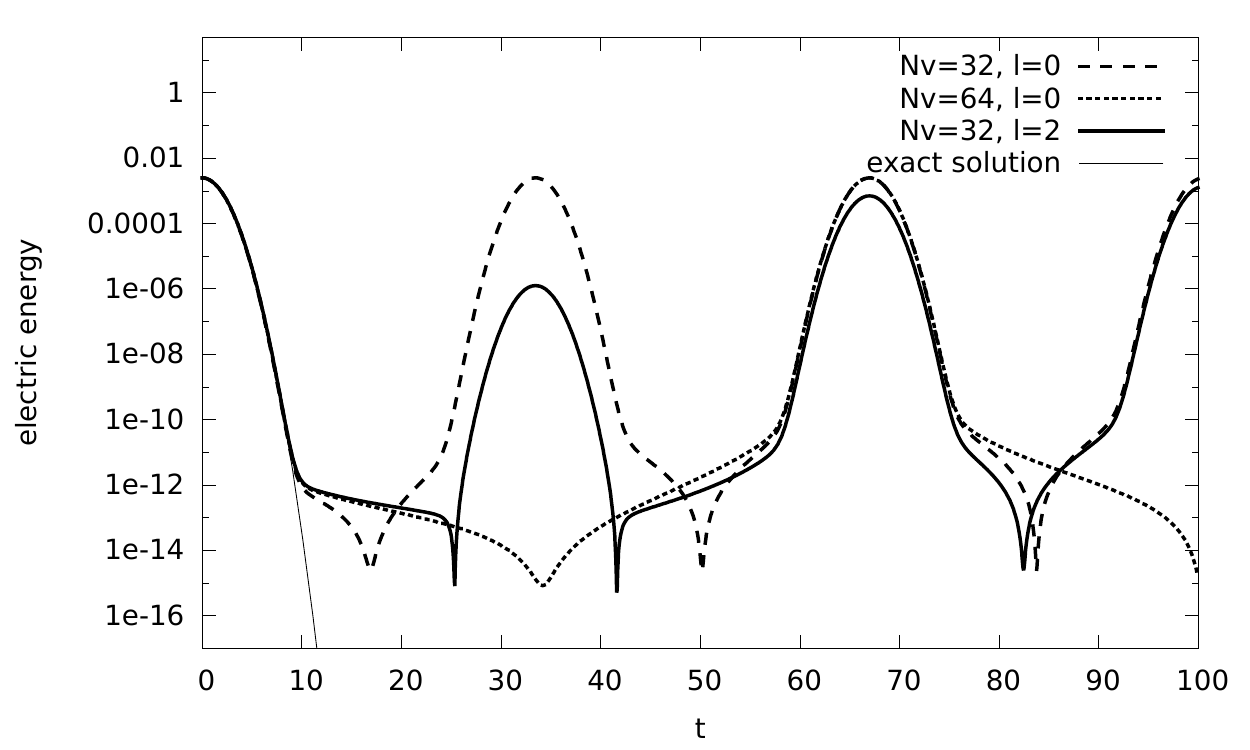}
	\centering \includegraphics[width=12cm]{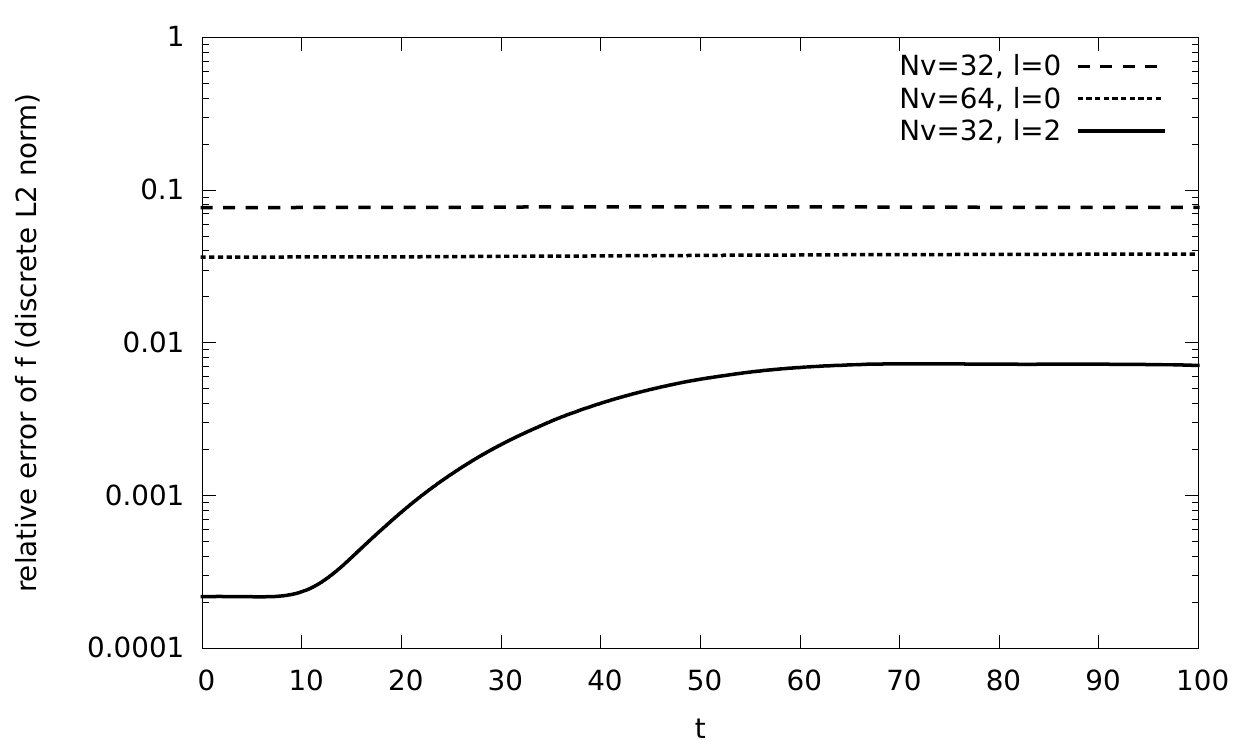}
	\caption{Recurrence phenomenon for the advection equation (top). Note that while there is no periodicity in the second order approximation a recurrence-like effect from the finite cell size is still visible. The (absolute) error as compared with the exact solution given in \eqref{eq:electric_energy_advection} in the discrete $L^2$ norm (bottom) behaves as expected. In all simulations $32$ cells and an approximation of order~$2$ (i.e. $\order=1$) have been employed in the space direction. The number of cells and the order of discretization in the velocity direction is indicated in the legend. In all computations $\dt = 0.05$ is used. \label{fig:recurrence_advection}}
      \end{figure}

      As mentioned before the recurrence phenomenon is also visible in the Landau damping problem. This is shown in Figure \ref{fig:recurrence_landau}.
      \begin{figure}[h]
	\centering \includegraphics[width=12cm]{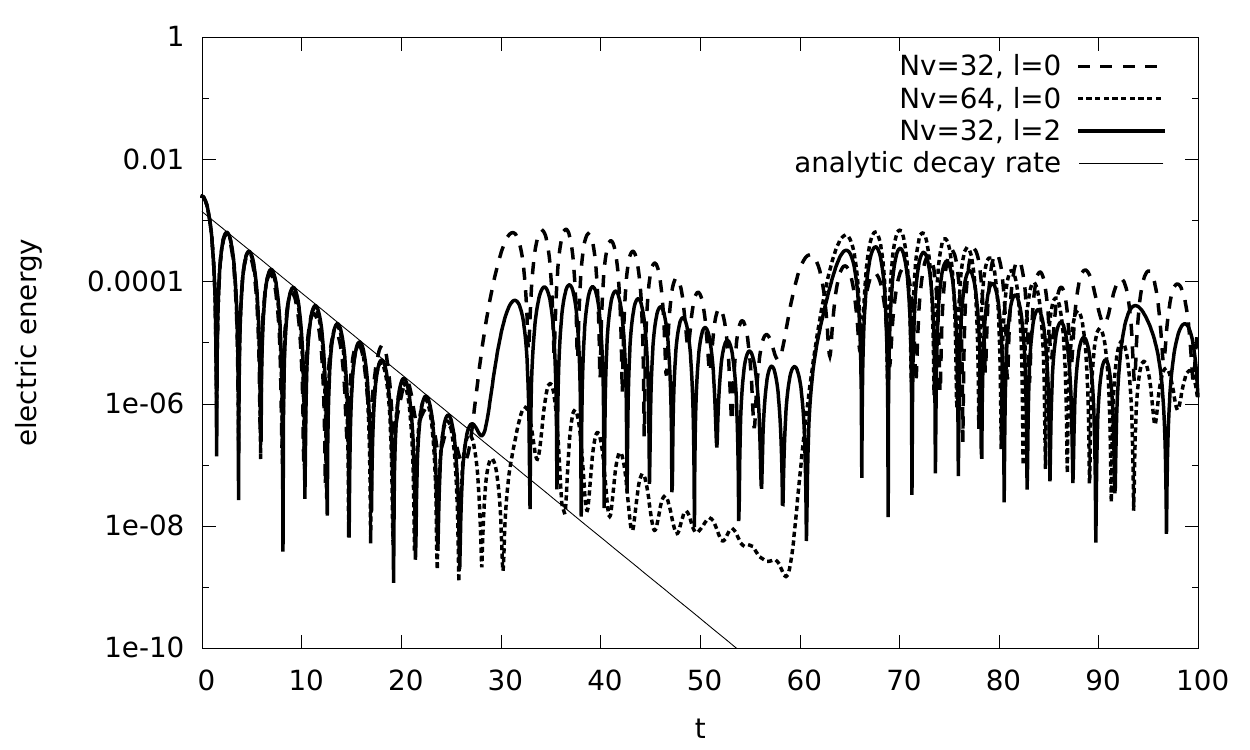}
	\caption{The recurrence phenomenon for Landau damping. Note that even though a higher order of approximation in the velocity direction improves the solution a recurrence-like effect is still visible. In all computations $\dt = 0.05$ is used. \label{fig:recurrence_landau}}
      \end{figure}

     \subsection{Order} \label{sec:order}
	We can investigate both, the order of convergence in time (i.e.~where the space error is small enough over the range of step sizes $\tau$ we are interested in) and the order of convergence in space (i.e.~where the step size is chosen small enough such that the time integration error is negligible). The order of the time integration has already been investigated in \cite{time_analysis}. Thus, we focus on the convergence order in space.
	
	Let us consider the Vlasov--Poisson equations in 1+1 dimensions together with the initial value
	\begin{equation*}
	  f_0(x,v)=\frac{1}{\sqrt{2 \pi}} \e^{-v^2/2} \big( 1+ \alpha\cos(0.5 x)\big).
	\end{equation*}
	This problem is called Landau damping. For $\alpha = 0.01$ the problem is called linear or weak Landau damping and for $\alpha=0.5$ it is referred to as strong or non-linear Landau damping. As can be seen, for example, in \cite{Crouseilles:2011,Filbet:2001} and \cite{Rossmanith:2011} Landau damping is a popular test problem for Vlasov codes.
	
	In our numerical simulations, all errors are computed with respect to a reference solution (such as to exclude unwanted effects from the time discretization). The reference solution uses $512$ cells with $\order=2$ and a time step of $\tau=0.1$. The results for strong Landau damping are given, up to order $3$, in Figure \ref{fig:convergence_space}. It can be seen that the accuracy improves with the desired order as the cell size decreases. Thus, the results are in good agreement with the theory developed in this paper.
	\begin{figure}[h]
	  \centering \includegraphics[width=12cm]{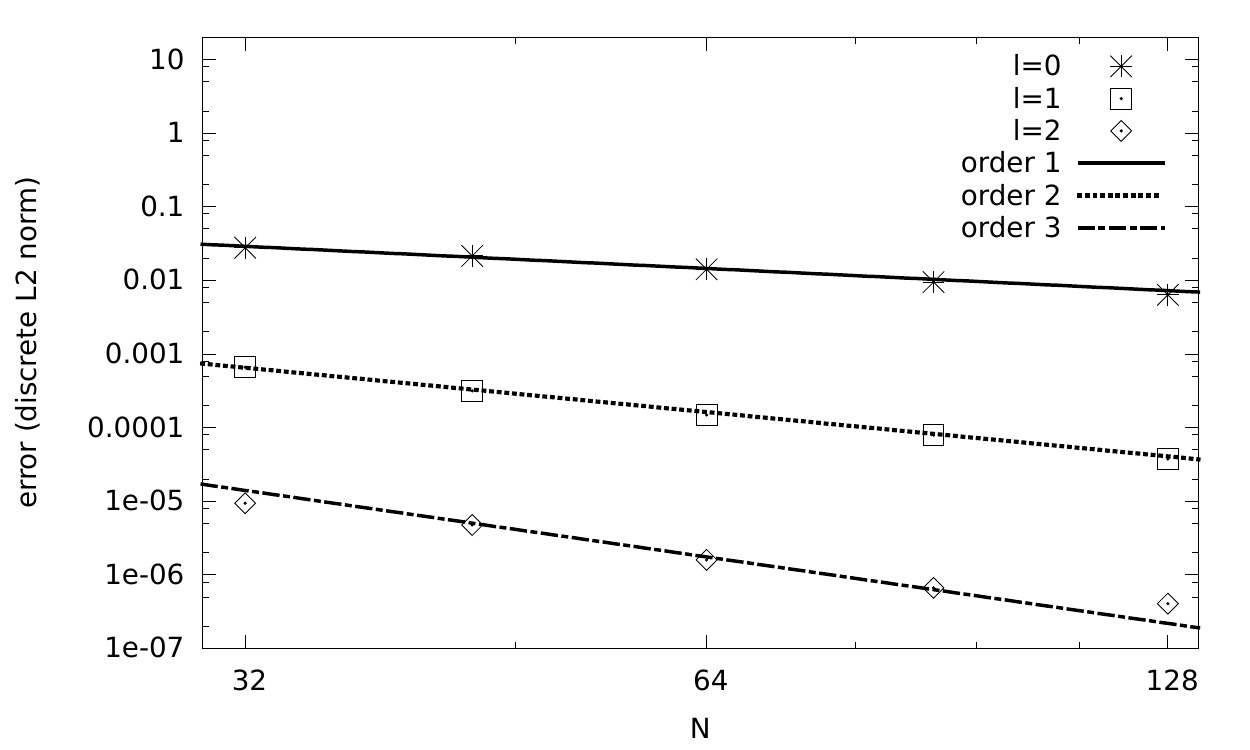}
	  \caption{Error of the particle density function $f(1,\cdot,\cdot)$ for non-linear Landau damping on the domain $[0,4 \pi]\times [-6,6]$ as a function of $N$, the number of cells in both the space and velocity direction. \label{fig:convergence_space}}
	\end{figure}

      \subsection{Landau damping} \label{sec:landau_damping}
	 The Landau damping problem has already been introduced in the previous section. In this section we are not interested in the desired order of the numerical algorithm but in the comparison with the exact solution of the Vlasov--Poisson equation. However, since an exact solution of the full Vlasov--Poisson equation is not known we will instead use a result that gives us the asymptotic decay rate $\gamma$ of the electric field in the case of weak Landau damping (see e.g. \cite{arber:2002}). Thus, we compare the decay of the energy stored in the electric field with
the graph of $\e^{-2 \gamma t}$, where $\gamma \approx 0.1533$. A number of such simulations have already been conducted (see e.g. \cite{zhou:2001}). However, due to the recurrence effect usually a large number of cells have to be employed in order to get accurate results for medium to long time intervals. For reference we note that \cite{zhou:2001} uses $N_x=N_v=1024$ whereas \cite{Heath:2011} uses up to $N_x=2000, N_v=1600$. The results of our simulation are shown in Figure \ref{fig:landau_damping} (the number of cells is $N_x=N_v=256$ for $\order=1$ and $N_x=N_v=128$ for $\order=2$). This experiment clearly shows that high-order approximations in space and velocity pay off.

	\begin{figure}
	  \centering \includegraphics[width=12cm]{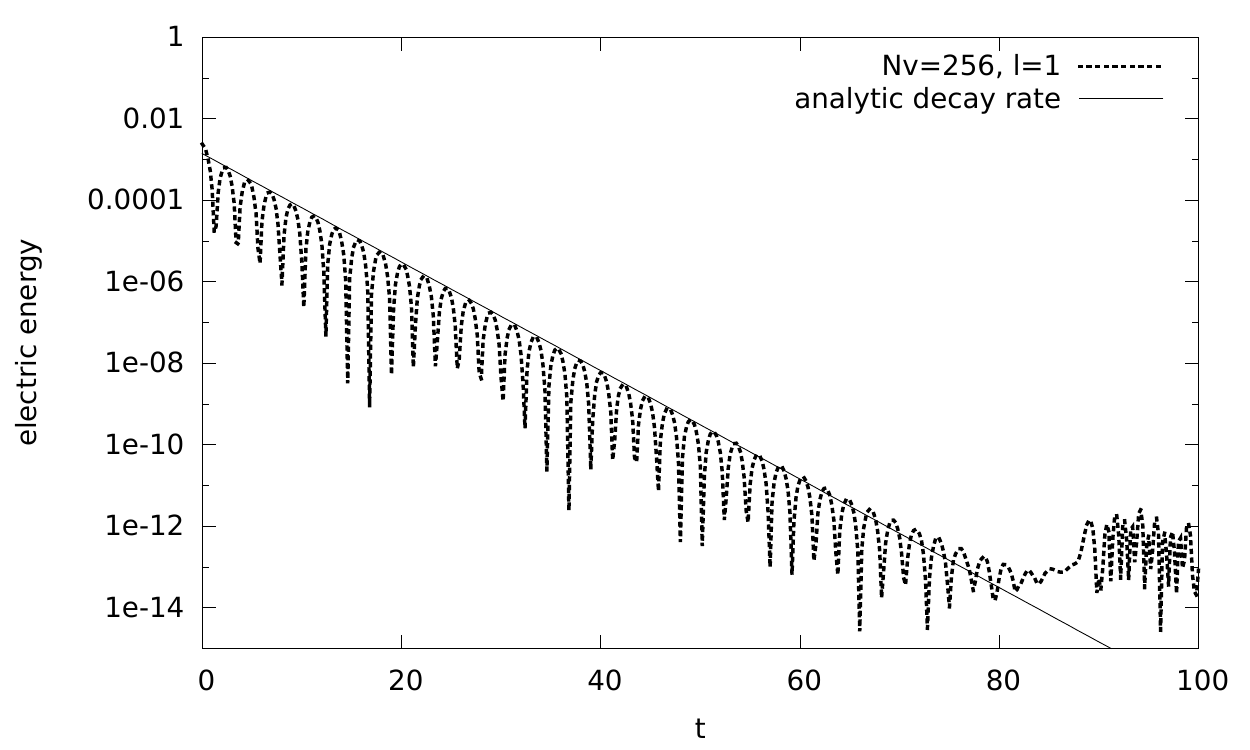}
	  \centering \includegraphics[width=12cm]{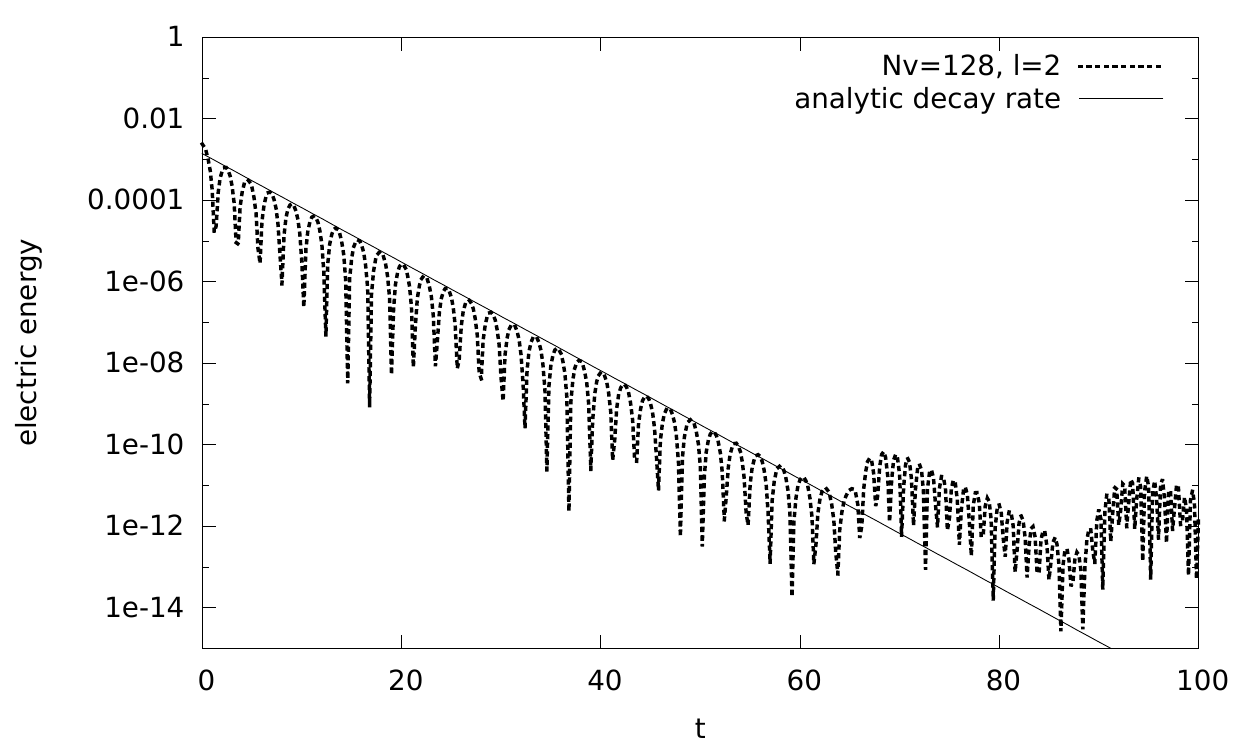}
	  \caption{ The decay of the electric field is shown for $N_x=N_v=256$, $\order=1$ (top) and $N_x=N_v=128$, $\order=2$ (bottom). In both cases a relatively large time step of $\tau=0.2$ is employed. \label{fig:landau_damping}}
	\end{figure}

      \subsection{Stability} \label{sec:numerical_stability}
	In advection dominated problems instabilities can occur if the numerical integration is not performed exactly. In \cite{morton:1988} this is shown for the Molenkamp--Crowley test problem, i.e.
	\begin{equation} \label{eq:cone_problem}
	  \left\{
	    \begin{aligned}
		      \partial_tf(t,x,y) &= 2 \pi (y \partial_x - x \partial_y ) f(t,x,y) \\
		      f(0,x,v)			&= f_0(x,y),
	    \end{aligned}
	  \right.
	\end{equation}
	where
	\[
	  f_0(x,y) = \begin{cases}
\cos^{2}(2\pi r) & r\leq\frac{1}{4}\\
0 & \text{else}
\end{cases}
	\]
	with $r^2=(x+\frac{1}{2})^2+y^2$. The solution travels along a circle with period $1$. We will solve the same problem using the algorithm presented in this paper and show that no instabilities occur if a quadrature rule of appropriate order is used. This results are given in Figure \ref{fig:numerical_stability}. Note that this is exactly what is expected based on the theoretical analysis done in section \ref{sec:vp11}. However, it is not true that such stability results hold for arbitrary schemes (see e.g. \cite{morton:1988}, where a finite element scheme of order $2$ is shown to be unstable for most quadrature rules).
	
	\begin{figure}
	  \centering {
	  	\includegraphics[width=14cm]{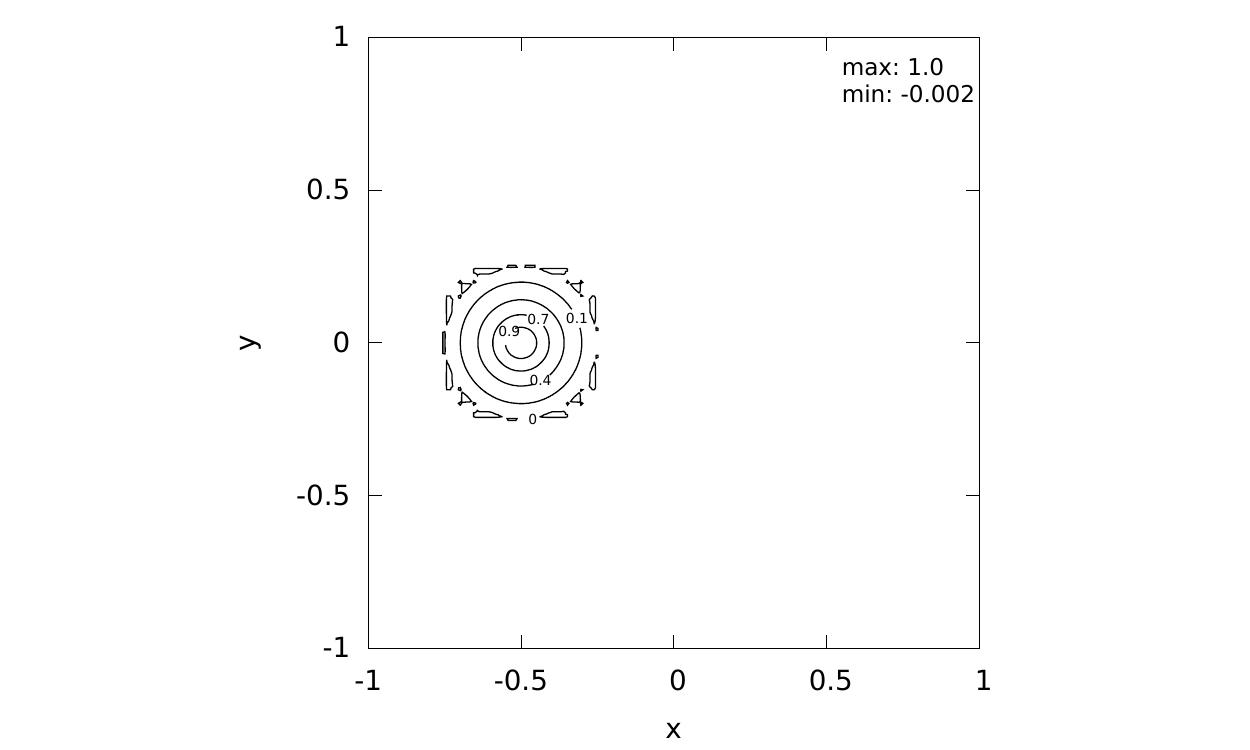}
	  	\includegraphics[width=14cm]{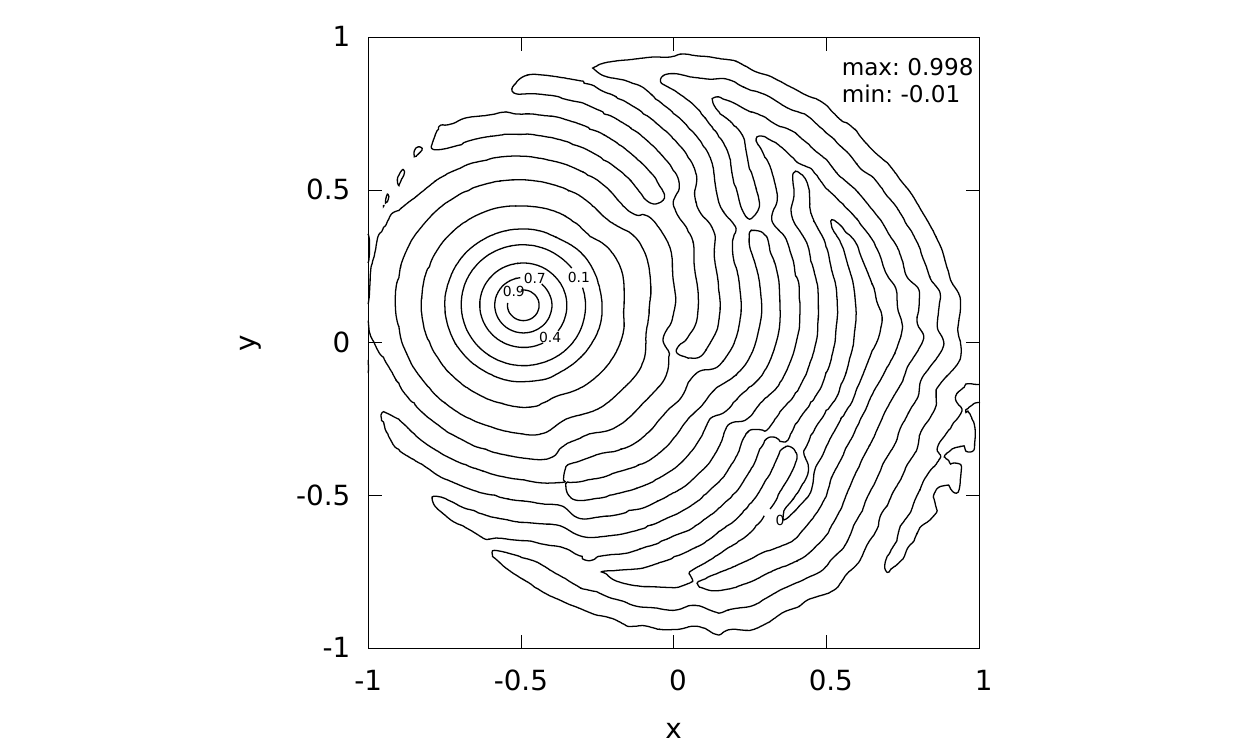}
	  }
	  \caption{Stability for the Molenkamp--Crowley test problem. The initial value is displayed at the top. The numerical solution after 60 revolutions with $\dt = 0.02$, $N_x=N_v=40$, and $\ell=2$ is shown at the bottom. As expected no numerical instabilities are observed. The negative values in the numerical solution are a consequence of the space discretization error and are propagated in space by the numerical algorithm (see the complex contour line of $0$ at the bottom). However, this fact has no influence on the stability of the scheme.\label{fig:numerical_stability}}
	\end{figure}

  \section{Conclusion}
  In the present paper we have extended the convergence analysis conducted in \cite{time_analysis} to the fully discretized case using a discontinuous Galerkin approximation in space. The results are only presented in case of the $1+1$ dimensional Vlasov--Poisson equation. However, we have given a short argument that the extension to multiple dimensions is although tedious in principle straightforward. In addition, we have presented a number of numerical simulations that investigate the behavior of the proposed algorithm. These simulations suggest that the algorithm has certain advantages over similar algorithms that employ a piecewise constant approximation in space and velocity.

  \bibliography{papers}

\begin{thebibliography}{10}

\bibitem{arber:2002}
{\sc T.D. Arber and R.~Vann}, {\em {A critical comparison of
  Eulerian-grid-based Vlasov solvers}}, J. Comput. Phys., 180 (2002),
  pp.~339--357.

\bibitem{Belli:2006}
{\sc E.A. Belli}, {\em {Studies of numerical algorithms for gyrokinetics and
  the effects of shaping on plasma turbulence}}, PhD thesis, {Princeton
  University}, 2006.

\bibitem{Besse:2005}
{\sc N.~Besse}, {\em {Convergence of a semi-Lagrangian scheme for the
  one-dimensional Vlasov-Poisson system}}, SIAM J. Numer. Anal., 42 (2005),
  pp.~350--382.

\bibitem{Besse:2008}
\leavevmode\vrule height 2pt depth -1.6pt width 23pt, {\em {Convergence of a
  high-order semi-Lagrangian scheme with propagation of gradients for the
  one-dimensional Vlasov-Poisson system}}, SIAM J. Numer. Anal., 46 (2008),
  pp.~639--670.

\bibitem{BesseMehr:2008}
{\sc N.~Besse and M.~Mehrenberger}, {\em {Convergence of classes of high-order
  semi-Lagrangian schemes for the Vlasov-Poisson system}}, Math. Comp., 77
  (2008), pp.~93--123.

\bibitem{Bostan:2009}
{\sc M.~Bostan and N.~Crouseilles}, {\em {Convergence of a semi-Lagrangian
  scheme for the reduced Vlasov-Maxwell system for laser-plasma interaction}},
  Numer. Math., 112 (2009), pp.~169--195.

\bibitem{Cheng:1976}
{\sc C.Z. Cheng and G.~Knorr}, {\em The integration of the {V}lasov equation in
  configuration space}, J. Comput. Phys., 22 (1976), pp.~330--351.

\bibitem{Crouseilles:2011}
{\sc N.~Crouseilles, E.~Faou, and M.~Mehrenberger}, {\em {High order
  Runge--Kutta--Nystr{\"o}m splitting methods for the Vlasov--Poisson
  equation}}.
\newblock \newline http://hal.inria.fr/inria-00633934/PDF/cfm.pdf.

\bibitem{time_analysis}
{\sc L.~Einkemmer and A.~Ostermann}, {\em {Convergence analysis of Strang
  splitting for Vlasov--type equations}}.
\newblock Preprint (arXiv:1207.2090), 2012.

\bibitem{Fahey:2008}
{\sc M.R. Fahey and J.~Candy}, {\em {GYRO}: A 5-d gyrokinetic-{M}axwell
  solver}, Proceedings of the ACM/IEEE SC2004 Conference,  (2008), p.~26.

\bibitem{Filbet:2001}
{\sc F.~Filbet and E.~Sonnendr{\"{u}}cker}, {\em {Comparison of Eulerian Vlasov
  solvers}}, Comput. Phys. Comm., 150 (2003), pp.~247--266.

\bibitem{glassey:1996}
{\sc R.T. Glassey}, {\em {The Cauchy Problem in Kinetic Theory}}, SIAM,
  Philadelphia, 1996.

\bibitem{Heath:2007}
{\sc R.E. Heath}, {\em {Analysis of the discontinuous Galerkin Method Applied
  to Collisionless Plasma Physics}}, PhD thesis, {The University of Texas at
  Austin}, 2007.

\bibitem{Heath:2011}
{\sc R.E. Heath, I.M. Gamba, P.J. Morrison, and C.~Michler}, {\em {A
  discontinuous Galerkin method for the Vlasov-Poisson system}}, J. Comput.
  Phys., 231 (2011), pp.~1140--1174.

\bibitem{Jahnke00}
{\sc T.~Jahnke and C.~Lubich}, {\em Error bounds for exponential operator
  splittings}, BIT, 40 (2000), pp.~735--744.

\bibitem{jenab:2011}
{\sc S.M.H. Jenab, I.~Kourakis, and H.~Abbasi}, {\em Fully kinetic simulation
  of ion acoustic and dust-ion acoustic waves}, Phys. Plasmas, 18 (2011),
  p.~073703.

\bibitem{Mangeney:2002}
{\sc A.~Mangeney, F.~Califano, C.~Cavazzoni, and P.~Travnicek}, {\em A
  numerical scheme for the integration of the {Vlasov-Maxwell} system of
  equations}, J. Comput. Phys., 179 (2002), pp.~495--538.

\bibitem{morton:1988}
{\sc K.W. Morton, A.~Priestley, and E.~S\"uli}, {\em {Stability of the
  Lagrange-Galerkin method with non-exact integration}}, Mod{\'e}l. Math. Anal.
  Num{\'e}r., 22 (1988), pp.~625--653.

\bibitem{phillips03}
{\sc G.M. Phillips}, {\em {Interpolation and Approximation by Polynomials}},
  Springer, 2003.

\bibitem{Respaud:2011}
{\sc T.~Respaud and E.~Sonnendr{\"{u}}cker}, {\em {Analysis of a new class of
  forward semi-Lagrangian schemes for the 1D Vlasov Poisson equations}}, Numer.
  Math., 118 (2011), pp.~329--366.

\bibitem{Rossmanith:2011}
{\sc J.A. Rossmanith and D.C. Seal}, {\em {A positivity-preserving high-order
  semi-Lagrangian discontinuous Galerkin scheme for the Vlasov-Poisson
  equations}}, J. Comput. Phys., 230 (2011), pp.~6203--6232.

\bibitem{shadrin2004twelve}
{\sc A.~Shadrin}, {\em Twelve proofs of the {M}arkov inequality}, in
  Approximation Theory: A volume dedicated to Borislav Bojanov, D.K. Dimitrov
  et~al., eds., Marin Drinov Acad. Publ. House, Sofia, 2004, pp.~233--298.

\bibitem{zhou:2001}
{\sc T.~Zhou, Y.~Guo, and C.W. Shu}, {\em {Numerical study on Landau damping}},
  Phys. D, 157 (2001), pp.~322--333.

\end{thebibliography}
  \bibliographystyle{siam}

\end{document}